\documentclass[11pt]{amsart}

\usepackage{amsfonts}
\usepackage{amscd}
\usepackage{amssymb}
\usepackage{amsthm}
\usepackage{amsmath}
\usepackage{comment}
\usepackage{epsfig}
\usepackage[all]{xy}
\usepackage[pdftex]{color}

\usepackage[colorlinks=true, pdfstartview=FitV, linkcolor=blue, citecolor=blue, urlcolor=blue]{hyperref}
\newcommand{\arxiv}[1]{\href{http://arxiv.org/abs/#1}{\tt arXiv:\nolinkurl{#1}}}

\newtheorem{Theorem}{Theorem}[section]
\newtheorem{Proposition}[Theorem]{Proposition}

\theoremstyle{definition}
\newtheorem{Example}[Theorem]{Example}
\newtheorem{Comment}[Theorem]{Comment}

\newtheorem{Definition}[Theorem]{Definition}

\def\asl{\widehat{\mathfrak{sl}}}
\def\g{\frak{g}}
\newcommand{\bz}{\Bbb{Z}}
\newcommand{\bc}{\mathbb{C}}

\newcommand{\agl}{\widehat{\mathfrak{gl}}}

\newcommand{\Hi}{{\bold H}}
\newcommand{\Cl}{{\bold Cl}}
\newcommand{\Fock}{{\bf F}}
\newcommand{\Bock}{{\bf B}}

\newcommand{\ch}{\mathrm{ch}}
\newcommand{\cW}{\mathcal{W}}

\def\fgl{\mathfrak{gl}}
\def\fsl{\mathfrak{sl}}

\def\End{\mathrm{End}}

\makeatletter
\renewcommand{\@makefnmark}{\mbox{\textsuperscript{}}}
\makeatother

\title{Notes on Fock space}

\author{Peter Tingley}

\address{Peter Tingley,
Dept. of Mathematics and Statistics,
Loyola University Chicago.}
\email{ptingley@luc.edu}

\date{\today}

\makeatletter
\def\@tocline#1#2#3#4#5#6#7{\relax
  \ifnum #1>\c@tocdepth 
  \else
    \par \addpenalty\@secpenalty\addvspace{#2}%
    \begingroup \hyphenpenalty\@M
    \@ifempty{#4}{%
      \@tempdima\csname r@tocindent\number#1\endcsname\relax
    }{%
      \@tempdima#4\relax
    }%
    \parindent\z@ \leftskip#3\relax \advance\leftskip\@tempdima\relax
    \rightskip\@pnumwidth plus4em \parfillskip-\@pnumwidth
    #5\leavevmode\hskip-\@tempdima
      \ifcase #1
       \or\or \hskip 1em \or \hskip 2em \else \hskip 3em \fi%
      #6\nobreak\relax
    \hfill\hbox to\@pnumwidth{\@tocpagenum{#7}}\par
    \nobreak
    \endgroup
  \fi}
\makeatother

\begin{document}

\begin{abstract}
These notes are intended as a fairly self-contained explanation of Fock space and various 
algebras that act on it, including a Clifford algebra, a Weyl algebra, an infinite rank matrix algebra, and an affine Kac-Moody algebra. We also discuss how the various algebras are related, and in particular describe the celebrated boson-fermion correspondence. We finish by briefly discussing a deformation of Fock space, which is a representation for the quantized universal enveloping algebra $U_q(\asl_\ell)$. 
\end{abstract}

\maketitle

\tableofcontents

\section{Introduction}

The term Fock space comes from particle physics, where it is the state space for a system of a variable number of elementary particles. There are two distinct types of elementary particles, bosons and fermions, and their Fock spaces look quite different. Fermionic Fock spaces are naturally representations of a Clifford algebra, where the generators correspond to adding/removing a particle in a given pure energy state. Bosonic Fock space is naturally a representation of a Weyl algebra. 

Here we focus on one example of this construction each for Bosons and Fermions. Our fermionic Fock space $\Fock$ corresponds to a system of fermionic particles with pure energy states indexed by $\bz+\frac12$. Our space $\Bock^{(0)}$ corresponds to a system of bosons with pure energy states indexed by $\bz_{>0}$, and our full bosonic Fock space $\Bock$ is a superposition of $\bz$ shifted copies of $\Bock^{(0)}$. There is a natural embedding of the Weyl algebra for $\Bock^{(0)}$ into a completion of the Clifford algebra for $\Fock$, which leads to the celebrated boson-fermion correspondence.

We begin by discussing $\Fock$ and $\Bock$ as vector spaces. We present various ways of indexing the standard basis of $\Fock$, and describe the boson-fermion correspondence as a vector space isomorphism between $\Fock$ and $\Bock$. Since $\Fock$ and $\Bock$ are both vector spaces with countable bases it is immediate that they are isomorphic, the isomorphism we present is only interesting in that it has nice properties with respect to the actions of the Clifford and Weyl algebras. A more intuitive development of the theory (and also how it is presented in \cite[Chapter 14]{Kac}) is probably to first notice the relationship between the Clifford and Weyl algebras, then derive the corresponding relationship between the vector spaces. However, we find it useful to first have the map explicitly described in an elementary way.

We next discuss the various algebras that naturally act on Fock space, and how these actions are related. These include an infinite rank matrix algebra and various affine Kac-Moody algebras, as well as the Weyl and Clifford algebras. We finish by presenting the Misra-Miwa action of the quantized universal enveloping algebra $U_q(\asl_\ell)$ on $\Fock \otimes_\bc \bc[q,q^{-1}].$ Understanding the relationship between the actions of these various algebras on $\Fock$ has proven very useful, see e.g. \cite[Chapter 14]{Kac}. It is really this representation theory that we are interested in, not the physics motivations. There will be very little physics beyond this introduction.

These notes are intended to be quite self-contained in the sense that they should be comprehensible independent of other references on Fock space. We do however refer to other sources for some important proofs. Much of our presentation loosely follows \cite[Chapter 14]{Kac}. Alternate references include \cite{DJM,KRR}. 

\subsection{Acknowledgments}

This notes were one consequence of a long series of discussions Arun Ram, so I would like to thank Arun for being so generous with his time. I would also like to thank Antoine Labelle, Tony Giaquinto, Gus Schrader and A.J. Tolland for helpful comments. Finally, I would like to thank my master's advisor Yuly Billig for having me read \cite{Kac} and all his help as I tried to understand it. 

\section{Fock space as a vector space}  \label{F_vect}

Fermionic Fock space $\Fock$ is an infinite dimensional vector space. It has a standard basis, which can be indexed by a variety of objects. In this section we discuss indexings by Maya diagrams, charged partitions, and normally ordered wedge products. 
Bosonic Fock space $\Bock$ is essentially a space of polynomials in infinitely many variables. This has a standard basis consisting of character polynomials, which are related to Schur symmetric functions. 
Schur symmetric functions are indexed by partitions, which leads to a natural bijection between the standard bases of $\Fock$ and $\Bock$, and hence a vector space isomorphism.

\subsection{Maya diagrams}

\begin{Definition} \label{Maya_def}
A \emph{Maya diagram} is a placement of a white or black bead at each position in $\bz+\frac12$, subject to the condition that all but finitely many positions $m<0$ have a black bead and all but finitely many positions $m>0$ have a white bead. For instance, 
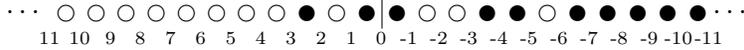
\setlength{\unitlength}{0.4cm}
\begin{center}
\begin{picture}(30,0.3)

\put(2.5,0){\ldots}
\put(26,0){\ldots}

\put(4.5,0){\circle{0.5}}
\put(5.5,0){\circle{0.5}}
\put(6.5,0){\circle{0.5}}
\put(7.5,0){\circle{0.5}}
\put(8.5,0){\circle{0.5}}
\put(9.5,0){\circle{0.5}}
\put(10.5,0){\circle{0.5}}
\put(11.5,0){\circle{0.5}}
\put(12.5,0){\circle*{0.5}}
\put(13.5,0){\circle{0.5}}
\put(14.5,0){\circle*{0.5}}

\put(15,-0.5){\line(0,1){1}}

\put(15.5,0){\circle*{0.5}}
\put(16.5,0){\circle{0.5}}
\put(17.5,0){\circle{0.5}}
\put(18.5,0){\circle*{0.5}}
\put(19.5,0){\circle*{0.5}}
\put(20.5,0){\circle{0.5}}
\put(21.5,0){\circle*{0.5}}
\put(22.5,0){\circle*{0.5}}
\put(23.5,0){\circle*{0.5}}
\put(24.5,0){\circle*{0.5}}
\put(25.5,0){\circle*{0.5}}

\put(14.8,-1){\tiny{0}}
\put(13.8,-1){\tiny{1}}
\put(12.8,-1){\tiny{2}}
\put(11.8,-1){\tiny{3}}
\put(10.8,-1){\tiny{4}}
\put(9.8,-1){\tiny{5}}
\put(8.8,-1){\tiny{6}}
\put(7.8,-1){\tiny{7}}
\put(6.8,-1){\tiny{8}}
\put(5.8,-1){\tiny{9}}
\put(4.6,-1){\tiny{10}}
\put(3.6,-1){\tiny{11}}

\put(15.6,-1){\tiny{-1}}
\put(16.6,-1){\tiny{-2}}
\put(17.6,-1){\tiny{-3}}
\put(18.6,-1){\tiny{-4}}
\put(19.6,-1){\tiny{-5}}
\put(20.6,-1){\tiny{-6}}
\put(21.6,-1){\tiny{-7}}
\put(22.6,-1){\tiny{-8}}
\put(23.6,-1){\tiny{-9}}
\put(24.4,-1){\tiny{-10}}
\put(25.4,-1){\tiny{-11}}

\end{picture}
\end{center}
\vspace{0.4cm}
It is sometimes convenient to think of the black beads as ``filled positions" and the white beads as ``empty positions" in a ``Dirac sea." 
Note that we label positions from right to left. 
This is to match the conventions for normally-ordered semi-infinite wedges, where the indices decrease to the right (see \S\ref{wedge_section} and \cite[\S14.9]{Kac}). 
\end{Definition}

\subsection{Charged partitions} \label{CP_section}

\begin{Definition}
A {\bf charged partition}  is a pair $(\lambda, h)$ consisting of a partition $\lambda$ and an integer $h$ (the charge). 
\end{Definition}

There is a natural bijection between Maya diagrams and charged partitions: draw a line above the Maya diagram by placing a segment sloping down and to the right over every white bead, and a segment sloping up and to the right over every black bead. See Figure \ref{partition_bij}. The result is the outer boundary of a partition, where by convention the parts of the partition are the lengths of the ``rows" of boxes going up and to the left. The charge is determined by superimposing axes with the origin above position $0$ in the Maya diagram, and so that far to the right the axis follows the diagram. The charge is the signed distance between the diagram and the axis far to the left. Equivalently, it is the number under the vertex of the partition.

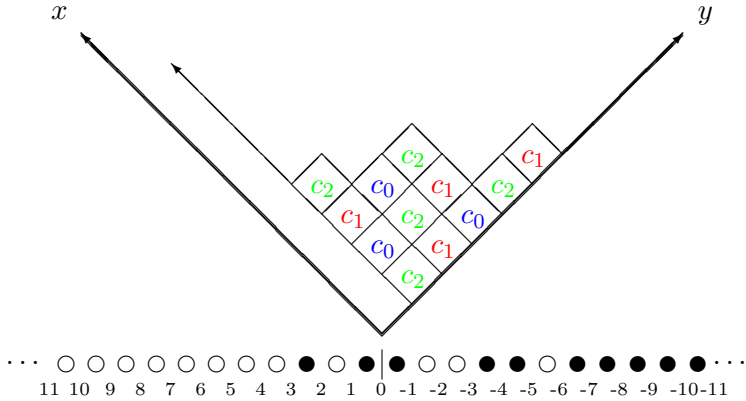
\begin{figure}

\setlength{\unitlength}{0.4cm}
\begin{center}
\begin{picture}(30,12)

\put(2.5,0){\ldots}
\put(26,0){\ldots}

\put(4.5,0){\circle{0.5}}
\put(5.5,0){\circle{0.5}}
\put(6.5,0){\circle{0.5}}
\put(7.5,0){\circle{0.5}}
\put(8.5,0){\circle{0.5}}
\put(9.5,0){\circle{0.5}}
\put(10.5,0){\circle{0.5}}
\put(11.5,0){\circle{0.5}}
\put(12.5,0){\circle*{0.5}}
\put(13.5,0){\circle{0.5}}
\put(14.5,0){\circle*{0.5}}

\put(15,-0.5){\line(0,1){1}}

\put(15.5,0){\circle*{0.5}}
\put(16.5,0){\circle{0.5}}
\put(17.5,0){\circle{0.5}}
\put(18.5,0){\circle*{0.5}}
\put(19.5,0){\circle*{0.5}}
\put(20.5,0){\circle{0.5}}
\put(21.5,0){\circle*{0.5}}
\put(22.5,0){\circle*{0.5}}
\put(23.5,0){\circle*{0.5}}
\put(24.5,0){\circle*{0.5}}
\put(25.5,0){\circle*{0.5}}

\put(9,9){\vector(-1,1){1}}
\put(9,9){\line(1,-1){1}}
\put(10,8){\line(1,-1){1}}
\put(11,7){\line(1,-1){1}}
\put(12,6){\line(1,1){1}}
\put(13,7){\line(1,-1){1}}
\put(14,6){\line(1,1){1}}
\put(15,7){\line(1,1){1}}
\put(16,8){\line(1,-1){1}}
\put(17,7){\line(1,-1){1}}
\put(18,6){\line(1,1){1}}
\put(19,7){\line(1,1){1}}
\put(20,8){\line(1,-1){1}}
\put(21,7){\line(1,1){1}}
\put(22,8){\line(1,1){1}}
\put(23,9){\line(1,1){1}}

\put(15,0.95){\line(1,1){10}}
\put(15,1){\vector(1,1){10}}
\put(15,1.05){\line(1,1){10}}
\put(4, 11.5){$x$}

\put(15,0.95){\line(-1,1){10}}
\put(15,1){\vector(-1,1){10}}
\put(15,1.05){\line(-1,1){10}}
\put(25.5, 11.5){$y$}

\put(16,2){\line(-1,1){4}}
\put(17,3){\line(-1,1){3}}
\put(18,4){\line(-1,1){3}}
\put(19,5){\line(-1,1){1}}
\put(20,6){\line(-1,1){1}}

\put(15,3){\line(1,1){3}}
\put(14,4){\line(1,1){3}}
\put(13,5){\line(1,1){1}}

\put(14.8,-1){\tiny{0}}
\put(13.8,-1){\tiny{1}}
\put(12.8,-1){\tiny{2}}
\put(11.8,-1){\tiny{3}}
\put(10.8,-1){\tiny{4}}
\put(9.8,-1){\tiny{5}}
\put(8.8,-1){\tiny{6}}
\put(7.8,-1){\tiny{7}}
\put(6.8,-1){\tiny{8}}
\put(5.8,-1){\tiny{9}}
\put(4.6,-1){\tiny{10}}
\put(3.6,-1){\tiny{11}}

\put(15.6,-1){\tiny{-1}}
\put(16.6,-1){\tiny{-2}}
\put(17.6,-1){\tiny{-3}}
\put(18.6,-1){\tiny{-4}}
\put(19.6,-1){\tiny{-5}}
\put(20.6,-1){\tiny{-6}}
\put(21.6,-1){\tiny{-7}}
\put(22.6,-1){\tiny{-8}}
\put(23.6,-1){\tiny{-9}}
\put(24.6,-1){\tiny{-10}}
\put(25.6,-1){\tiny{-11}}

\put(12.3,5.7){{ \color{green} $c_2$}}
\put(13.3,4.7){{ \color{red} $c_1$}}
\put(14.3,3.7){{ \color{blue} $c_0$}}
\put(14.3,5.7){{ \color{blue} $c_0$}}
\put(15.3,6.7){{ \color{green} $c_2$}}
\put(15.3,4.7){{ \color{green} $c_2$}}
\put(15.3,2.7){{ \color{green} $c_2$}}
\put(16.3,3.7){{ \color{red} $c_1$}}
\put(16.3,5.7){{ \color{red} $c_1$}}
\put(17.3,4.7){{ \color{blue} $c_0$}}
\put(18.3,5.7){{ \color{green} $c_2$}}
\put(19.3,6.7){{ \color{red} $c_1$}}

\end{picture}
\end{center}

\vspace{0.4cm}

\caption{The bijection between Maya diagrams and charged partitions from \S\ref{CP_section}. The black beads in the Maya diagram correspond to the positions where the rim of the partition slopes up and to the right. The parts of $\lambda$ are the lengths of all the finite ``rows" of boxes sloping up at to the left, so here $\lambda=(4,3,3,1,1)$. The charge is the horizontal position of the vertex of the partition, so here the charge is $h=-1$. Later on in \S\ref{F_aff} we will choose a ``level" $\ell$ and color the squares of $\lambda$ with $c_{\bar i}$ for residues $\bar i$ mod $\ell$, where the color of a box $b$ is the position of that box in the horizontal direction mod $\ell$.\label{partition_bij}}
\end{figure}

\subsection{Semi-infinite wedges} \label{wedge_section}

\begin{Definition}
Let $V_{\bz+\frac12}$ be the $\bc$-vector space with basis $\{e_m \}_{m \in \bz+\frac12}$.
\end{Definition}

\begin{Definition}
A {semi-infinite wedge} is a (formal) infinite wedge $e_{m_1} \wedge e_{m_2} \wedge e_{m_3} \wedge \cdots$. This can be thought of as living inside a formal semi-infinite wedge space  $V_{\bz+\frac12} \wedge V_{\bz+\frac12} \wedge V_{\bz+\frac12} \cdots$. 
\end{Definition}

As usual, ``wedge" is anti-commutative. So, for example,
$$e_{2.5} \wedge e_{0.5} \wedge e_{-1.5} \wedge e_{-2.5} \cdots =  -e_{0.5} \wedge e_{2.5} \wedge e_{-1.5} \wedge e_{-2.5} \cdots $$

\begin{Definition}
A semi-infinite wedge $e_{m_1} \wedge e_{m_2} \wedge e_{m_3} \wedge \cdots$ is called \emph{ normally ordered} if $m_1 > m_2 >m_3 > \cdots$.  
\end{Definition}

\begin{Definition}
A semi-infinite wedge $e_{m_1} \wedge e_{m_2} \wedge e_{m_3} \wedge \cdots$  is called \emph{regular} if $m_{k+1} = m_k-1$ for all sufficiently large $k$.
\end{Definition}

There is a bijection between regular normally ordered semi-infinite wedges and Maya diagrams which takes $e_{m_1} \wedge e_{m_2} \wedge \cdots$ to the Maya diagram with black beads exactly in positions $m_1, m_2, \ldots$. For instance, 
\begin{equation}
 e_{2.5} \wedge e_{0.5} \wedge e_{-0.5} \wedge e_{-3.5} \wedge e_{-4.5} \wedge e_{-6.5} \wedge e_{-7.5} \wedge e_{-8.5} \cdots
 \end{equation}
corresponds to the Maya diagram shown in Definition \ref{Maya_def} and Figure \ref{partition_bij}. 

\subsection{Fermionic Fock space $\Fock$}
The fermionic Fock space $\Fock$ is the ${\Bbb C}$-vector space with basis indexed by any one of the following:
\begin{enumerate}

\item Maya Diagrams.

\item Charged partitions.

\item Regular normally-ordered semi-infinite wedges.
\end{enumerate}
As in \S\ref{CP_section} and \S\ref{wedge_section} these three sets are naturally in bijection, and we take the point of view that they all index the same basis of $\Fock$. We most often refer to a basis element using the corresponding charged partition, using the notation $|\lambda,h \rangle$ to denote the standard basis element of $F$ corresponding to the charged partition $(\lambda,h)$.

The \emph{charge} $h$ part of $\Fock$ is
\begin{equation*}
\Fock^{(h)}:= span \{\text{charged partitions with charge } h\}.
\end{equation*}
In particular, ordinary partitions index a basis for $\Fock^{(0)}$.
As a vector space, $\Fock= \oplus_{h \in \bz} \Fock^{(h)}$.

\subsection{Bosonic Fock space $\Bock$} \label{ss:BF}
The \emph{bosonic Fock space} $\Bock$ is 
\begin{equation*}
\Bock:= \bc[x_1,x_2, x_3, \ldots; q,q^{-1}].
\end{equation*}
The \emph{charge} $h$ part of $\Bock$ is 
\begin{equation*}
\Bock^{(h)}:= q^h \bc[x_1, x_2, \ldots].\end{equation*}
As vector spaces, $\Bock = \oplus_{h \in \bz} \Bock^{(h)}$. 

\subsection{The boson-fermion correspondence as a map of vector spaces} \label{BF-vect}

\begin{Definition}
Fix a partition $\lambda$. A \emph{column-strict filling} of $\lambda$ is a function $t$ from the set of boxes in the diagram of $\lambda$ to $\bz_{>0}$ which is weakly increasing along rows and strictly increasing along columns. See Figure \ref{semi-standard}.
\end{Definition}

\begin{figure}

\setlength{\unitlength}{0.4cm}
\begin{center}
\begin{picture}(30,8)

\put(15,0){\vector(1,1){6}}
\put(15,0){\vector(-1,1){6}}

\put(16,1){\line(-1,1){4}}
\put(17,2){\line(-1,1){2}}
\put(18,3){\line(-1,1){2}}
\put(19,4){\line(-1,1){1}}

\put(14,1){\line(1,1){4}}
\put(13,2){\line(1,1){3}}
\put(12,3){\line(1,1){1}}
\put(11,4){\line(1,1){1}}

\put(14.8,0.6){$1$}
\put(13.8,1.6){$1$}
\put(12.8,2.6){$3$}
\put(11.8,3.6){$3$}
\put(15.8,1.6){$3$}
\put(14.8,2.6){$3$}
\put(16.8,2.6){$4$}
\put(15.8,3.6){$5$}
\put(17.8,3.6){$5$}

\end{picture}
\end{center}

\caption{\label{semi-standard} A column-strict filling of $\lambda=(4,2,2,1)$. The function $t$ takes $b$ to the integer in box $b$. Recall that ``rows" slope up and to the left, and ``columns" slope up and to the right.}
\end{figure}
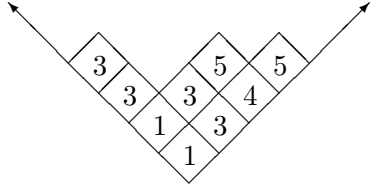

\begin{Definition} Let $\lambda$ be a partition. The \emph{Schur symmetric function} $s_{\lambda}$ in infinitely many variables $y_1, y_2, \ldots$ corresponding to $\lambda$ is
\begin{equation}
s_{\lambda}({\bf y}) := \sum_{\tiny \begin{array}{c} t \text{ a column strict} \\ \text{filling of } \lambda \end{array}} \prod_{\text{boxes } b \text{ of } \lambda} y_{t(b)}. 
\end{equation}
\end{Definition}

\begin{Comment} 
The Schur symmetric functions are symmetric in the sense that they are invariant under permutations of the variables $y_i$ (this can be shown using combinatorics \cite[Theorem 7.10.2]{Sta}, or by appealing to representation theory \cite[\S I3)]{Mac}). They are polynomials in the sense that, if all but finitely many of the variables are set to 0, the result is a polynomial. They have many important properties, most of which arise because (once all but the first $n$ variables are set to 0) they are the characters of irreducible representations of $\fgl_n$. 
\end{Comment}

\begin{Definition}
For $k \in {\Bbb Z}_{>0}$, let $p_k= y_1^k+y_2^k+\cdots$ be the $k^{th}$ power symmetric function.
\end{Definition}

\begin{Definition}
The character polynomial $\chi_{\lambda}$ is the unique polynomial such that 
\begin{equation*}
\chi_{\lambda}(p_1,\frac{1}{2} p_2, \frac{1}{3} p_3, \frac{1}{4} p_4 \ldots)= s_{\lambda}.
\end{equation*}
\end{Definition}

\begin{Comment}
To find $\chi_{\lambda}$, one may set $y_j=0$ for all $j$ larger than the longest column of $\lambda$. As well, $\chi_{\lambda}$ cannot depend on $p_k$ for any $k $ bigger than the number of boxes in $\lambda$ (since $s_\lambda$ is homogeneous of degree $<k$), so finding any given $\chi_{\lambda}$ is a finite problem.
\end{Comment}

\begin{Proposition} \label{Fock-sym} 
There is an isomorphism of vector spaces $\sigma: \Fock \rightarrow \Bock$ given by 
$$\sigma(|\lambda,h \rangle)= q^h \chi_{\lambda}(x_1, x_2, \ldots)$$
for any charged partition $(\lambda, h)$. 
This restricts to an isomorphism $\Fock^{(h)} \rightarrow B^{(h)}$ for all $h \in \bz$.
\end{Proposition}

\begin{proof}
As in \cite[\S I3]{Mac}, Schur polynomials are a basis for the space of all symmetric functions in infinitely many variables, from which it follows that $\{ q^h \chi_{\lambda} \}$ is a basis for $\Bock$. By definition, the set of all charged partitions is a basis for $\Fock$. Thus each restriction $\sigma : \Fock^{(h)} \rightarrow \Bock^{(h)}$ is a vector space isomorphism, and so is the full map $\sigma: \Fock^{(h)} \rightarrow \Bock^{(h)}$. 
\end{proof}

\subsection{The inner product} \label{form_section}
It is often convenient to put an inner product on $\Fock$, where we declare $\{ |\lambda, h \rangle \}$ to be an orthonormal basis.

\section{Fock space as a representation} \label{F_rep}

\subsection{$\Fock$ as a representation of a Clifford algebra} \label{F_ferm}

The Clifford algebra is the associative algebra $\Cl$ generated by $\psi_m, \psi^*_m$ for $m \in \bz+\frac12$ subject to the relations
\begin{align}
& \psi_n\psi_m+\psi_m \psi_n=0 ,\\
& \psi^*_n\psi^*_m+\psi^*_m \psi_n^*=0, \\
& \label{Cliff3} \psi_n\psi^*_m+\psi^*_m \psi_n= \delta_{m,n}. 
\end{align}
One can check that $\Cl$ acts on Fock space as follows: Use the description of $\Fock$ in terms of semi-infinite wedges from \S\ref{wedge_section}. For $m \in \bz+\frac12$ and $v$ a regular semi-infinite wedge, define
\begin{align}
\psi_m \cdot v=  e_m \wedge v.
\end{align}
The generator $\psi_n^*$ acts on $\Fock$ as the adjoint of $\psi_n^*$ with respect to the inner product from \S\ref{form_section}. Explicitly, for $v$ a regular semi-infinite wedge, 
\begin{equation}
\psi_m^* \cdot v=\begin{cases}
0 \quad \text{ if } e_m \text{ does not appear as a factor in } v, \\
v' \quad \text{if } v \text{ can be expressed as } e_m \wedge v' \text{ for a wedge } v'.
\end{cases}
\end{equation}
One can obtain any regular semi-infinite wedge $v$ from any other by applying a finite number of operators $\psi_m$ and $\psi_m^*$. Also, given any linear combination of regular semi-infinite wedges, one can project onto a single term by applying enough $\psi_n \psi_n^*$ and $\psi^*_n \psi_n$. Together this implies that $\Fock$ is irreducible as a representation of $\Cl$. 

\begin{Definition}
$\widetilde{\Cl}$ is the completion of $\Cl$ in the topology generated by the open sets $X+ \Cl \psi_m^*$ and $X+\Cl \psi_{-m}$ for all $X \in \Cl$, $m>0$. Explicitly, an element of $\widetilde \Cl$ is an infinite sum
\begin{equation}
z+ \sum_{m<N} X_m \psi_m + \sum_{m>M} Y_m \psi^*_m, 
\end{equation}
where $z \in \bc$,  $N,M \in \bz$, and all $X_m$, $Y_m$ are elements of $\Cl$.
\end{Definition}

The completion $\widetilde{\Cl}$ acts in a well-defined way on $\Fock$, since, for each regular semi-infinite wedge $v$, all but finitely many of $\psi^*_m, \psi_{-m}$ for $m>0$ act by 0 on $v$. 

\begin{Comment}
When it seems prudent to distinguish $X \in \widetilde{\Cl}$ from its action of $\Fock$, we use the notation $\pi_\Fock(X)$ to denote the action.
\end{Comment}

\subsection{$\Bock$ as a representation of a Weyl algebra} \label{F_bos}

\begin{Definition} \label{def:H}
The infinite dimensional Heisenberg algebra $\Hi$ is generated by $\alpha_k$ for $k \in \bz_{\neq 0}$ and a central element $c$ subject to the relations:
\begin{align*}
\alpha_j \alpha_k -\alpha_k \alpha_j=  j \delta_{j,-k} c.
\end{align*}
\end{Definition}

\begin{Definition}
The Weyl algebra $\cW$ is
\begin{equation*}
\cW:= \Hi/(c-1).
\end{equation*}
\end{Definition}

\begin{Proposition} \label{W-acts} 
The Weyl algebra $\cW$ acts on $\Bock$ by the follows formulas. Each $B^{(h)}$ is preserved under this action and forms an irreducible representation for $\cW$.
\begin{align*}
&\alpha_k \rightarrow  \frac{\partial}{\partial x_{k}}  &&\text{ if } k > 0, \\
& \alpha_k \rightarrow \text{ multiplication by } -k x_{-k} && \text{ if } k <0.
\end{align*} 
\end{Proposition}

\begin{proof}
A direct calculation shows that, defining the actions of $\alpha_k$ and $\alpha_{-k}$ as above and allowing $c$ to act as multiplication by $1$, the defining relations in Definition \ref{def:H} hold. So this gives an action of $\cW$ on $\Bock$. Since none of these operators affect the power of $q$ in any monomial each $B^{(h)}$ is preserved. 
To see that each $B^{(h)}$ is irreducible as a $\cW$ module, notice that
\begin{itemize}
\item Any non-zero $p \in B^{(h)}$ can be brought to a non-zero multiple of $q^h$ by applying a sequence of partial derivatives, which are the actions of $\alpha_k$ for $k >0$.

 \item Any monomial in $B^{(h)}$ can be constructed from $q^h$ by applying a sequence off multiplications by the generators which, up to scalar, are the actions of $\alpha_k$ for $k <0$. 
\end{itemize}
\end{proof}

\begin{Comment}
When it seems prudent to distinguish $Y \in \cW$ from its action of $\Bock$ we use the notation $\pi_\Bock(Y)$ to denote the action. 
\end{Comment}

\subsection{The boson-fermion correspondence} \label{Boson-Fermion-correspondence}

Proposition \ref{Fock-sym} gives an isomorphism of vector spaces $\sigma: \Fock \rightarrow \Bock$. This isomorphism was chosen because it reveals an important relationship between the algebras $\Cl$ and $\cW$, which is known as the boson-fermion correspondence. We now explain that relationship, referring to \cite[Chapter 14]{Kac} for rigorous proofs.

\begin{Definition}
For $Y \in \cW$, let $\pi_\Fock(Y)= \sigma^{-1} \circ \pi_\Bock(Y) \circ \sigma$. That is, $\pi_\Fock(Y)$ is the operator on $\Fock$ induced from $\pi_\Bock(Y)$ by the vector space isomorphism $\sigma: \Fock \rightarrow \Bock$.
\end{Definition}

\begin{Proposition} \label{finding-Bosons} (see \cite[\S14.10]{Kac}) 
For all $k \in \bz_{\neq 0}$,
\begin{align*}
 \pi_\Fock(\alpha_k) = \pi_\Fock \sum_{m \in \bz+\frac12} \psi_{m} \psi^*_{m+k}.
\end{align*}
\end{Proposition}

\begin{Comment} \label{com:a0}
In \cite[\S14.10]{Kac} Kac also defines $\alpha_0$. Proposition \ref{finding-Bosons} can be extended to include $\alpha_0$, but one must be careful, since $\displaystyle \sum_{m \in \bz+\frac12} \psi_{m} \psi^*_{m}$ would acts on any $v\in \Fock$ as multiplication by $\infty$. Instead, $ \alpha_0$ corresponds to $\displaystyle \sum_{m >0 } \psi_{m} \psi^*_{m} - \sum_{m <0} \psi^*_{m} \psi_{m}$, which acts on $F^{(h)}$ simply as multiplication by $h$. 
\end{Comment}

\begin{Example} \label{bosons-acting} The action of $\alpha_k$ on the standard basis of $\Fock$ has a nice description using Maya diagrams. If $v$ is a Maya diagram, then
$\alpha_k v$ is the sum of $(-1)^{j(v,v')} v'$ over all $v'$ obtained from $v$ by moving a single black bead to the right exactly $k$ places, where $j(v,v')$ is the number of black beads that are ``jumped" by the bead that moves. If $k$ is negative, this means the beads move to the left instead. For example,
\begin{equation}
\begin{aligned}
\alpha_{2} ( & \hspace{-1cm}
 \setlength{\unitlength}{0.4cm} \begin{picture}(30,0.3)

\put(2.5,0){\ldots}
\put(26,0){\ldots}

\put(4.5,0){\circle{0.5}}
\put(5.5,0){\circle{0.5}}
\put(6.5,0){\circle{0.5}}
\put(7.5,0){\circle{0.5}}
\put(8.5,0){\circle{0.5}}
\put(9.5,0){\circle{0.5}}
\put(10.5,0){\circle{0.5}}
\put(11.5,0){\circle{0.5}}
\put(12.5,0){\circle*{0.5}}
\put(13.5,0){\circle{0.5}}
\put(14.5,0){\circle*{0.5}}

\put(15,-0.5){\line(0,1){1}}

\put(15.5,0){\circle*{0.5}}
\put(16.5,0){\circle{0.5}}
\put(17.5,0){\circle{0.5}}
\put(18.5,0){\circle*{0.5}}
\put(19.5,0){\circle*{0.5}}
\put(20.5,0){\circle{0.5}}
\put(21.5,0){\circle*{0.5}}
\put(22.5,0){\circle*{0.5}}
\put(23.5,0){\circle*{0.5}}
\put(24.5,0){\circle*{0.5}}
\put(25.5,0){\circle*{0.5}}

\end{picture} \hspace{-1cm}  ) = \\ \\
 - &  \hspace{-1cm} \setlength{\unitlength}{0.4cm}  \begin{picture}(30,0.3)

\put(2.5,0){\ldots}
\put(26,0){\ldots}

\put(4.5,0){\circle{0.5}}
\put(5.5,0){\circle{0.5}}
\put(6.5,0){\circle{0.5}}
\put(7.5,0){\circle{0.5}}
\put(8.5,0){\circle{0.5}}
\put(9.5,0){\circle{0.5}}
\put(10.5,0){\circle{0.5}}
\put(11.5,0){\circle{0.5}}
\put(12.5,0){\circle*{0.5}}
\put(13.5,0){\circle{0.5}}
\put(14.5,0){\circle{0.5}}

\put(15,-0.5){\line(0,1){1}}

\put(15.5,0){\circle*{0.5}}
\put(16.5,0){\circle*{0.5}}
\put(17.5,0){\circle{0.5}}
\put(18.5,0){\circle*{0.5}}
\put(19.5,0){\circle*{0.5}}
\put(20.5,0){\circle{0.5}}
\put(21.5,0){\circle*{0.5}}
\put(22.5,0){\circle*{0.5}}
\put(23.5,0){\circle*{0.5}}
\put(24.5,0){\circle*{0.5}}
\put(25.5,0){\circle*{0.5}}

\end{picture} \\
+ &  \hspace{-1cm} \setlength{\unitlength}{0.4cm} \begin{picture}(30,0.3)

\put(2.5,0){\ldots}
\put(26,0){\ldots}

\put(4.5,0){\circle{0.5}}
\put(5.5,0){\circle{0.5}}
\put(6.5,0){\circle{0.5}}
\put(7.5,0){\circle{0.5}}
\put(8.5,0){\circle{0.5}}
\put(9.5,0){\circle{0.5}}
\put(10.5,0){\circle{0.5}}
\put(11.5,0){\circle{0.5}}
\put(12.5,0){\circle*{0.5}}
\put(13.5,0){\circle{0.5}}
\put(14.5,0){\circle*{0.5}}

\put(15,-0.5){\line(0,1){1}}

\put(15.5,0){\circle{0.5}}
\put(16.5,0){\circle{0.5}}
\put(17.5,0){\circle*{0.5}}
\put(18.5,0){\circle*{0.5}}
\put(19.5,0){\circle*{0.5}}
\put(20.5,0){\circle{0.5}}
\put(21.5,0){\circle*{0.5}}
\put(22.5,0){\circle*{0.5}}
\put(23.5,0){\circle*{0.5}}
\put(24.5,0){\circle*{0.5}}
\put(25.5,0){\circle*{0.5}}

\end{picture} \\
- &  \hspace{-1cm} \setlength{\unitlength}{0.4cm} \begin{picture}(30,0.3)

\put(2.5,0){\ldots}
\put(26,0){\ldots}

\put(4.5,0){\circle{0.5}}
\put(5.5,0){\circle{0.5}}
\put(6.5,0){\circle{0.5}}
\put(7.5,0){\circle{0.5}}
\put(8.5,0){\circle{0.5}}
\put(9.5,0){\circle{0.5}}
\put(10.5,0){\circle{0.5}}
\put(11.5,0){\circle{0.5}}
\put(12.5,0){\circle*{0.5}}
\put(13.5,0){\circle{0.5}}
\put(14.5,0){\circle*{0.5}}

\put(15,-0.5){\line(0,1){1}}

\put(15.5,0){\circle*{0.5}}
\put(16.5,0){\circle{0.5}}
\put(17.5,0){\circle{0.5}}
\put(18.5,0){\circle{0.5}}
\put(19.5,0){\circle*{0.5}}
\put(20.5,0){\circle*{0.5}}
\put(21.5,0){\circle*{0.5}}
\put(22.5,0){\circle*{0.5}}
\put(23.5,0){\circle*{0.5}}
\put(24.5,0){\circle*{0.5}}
\put(25.5,0){\circle*{0.5}}

\end{picture} 
\end{aligned}
\end{equation}
and
\begin{equation}
\begin{aligned}
\alpha_{-4} ( & \hspace{-1cm}
 \setlength{\unitlength}{0.4cm} \begin{picture}(30,0.3)

\put(2.5,0){\ldots}
\put(26,0){\ldots}

\put(4.5,0){\circle{0.5}}
\put(5.5,0){\circle{0.5}}
\put(6.5,0){\circle{0.5}}
\put(7.5,0){\circle{0.5}}
\put(8.5,0){\circle{0.5}}
\put(9.5,0){\circle{0.5}}
\put(10.5,0){\circle{0.5}}
\put(11.5,0){\circle{0.5}}
\put(12.5,0){\circle*{0.5}}
\put(13.5,0){\circle{0.5}}
\put(14.5,0){\circle*{0.5}}

\put(15,-0.5){\line(0,1){1}}

\put(15.5,0){\circle*{0.5}}
\put(16.5,0){\circle{0.5}}
\put(17.5,0){\circle{0.5}}
\put(18.5,0){\circle*{0.5}}
\put(19.5,0){\circle*{0.5}}
\put(20.5,0){\circle{0.5}}
\put(21.5,0){\circle*{0.5}}
\put(22.5,0){\circle*{0.5}}
\put(23.5,0){\circle*{0.5}}
\put(24.5,0){\circle*{0.5}}
\put(25.5,0){\circle*{0.5}}

\end{picture} \hspace{-1cm}  ) = \\ \\
 &  \hspace{-1cm} \setlength{\unitlength}{0.4cm}  \begin{picture}(30,0.3)

\put(2.5,0){\ldots}
\put(26,0){\ldots}

\put(4.5,0){\circle{0.5}}
\put(5.5,0){\circle{0.5}}
\put(6.5,0){\circle{0.5}}
\put(7.5,0){\circle{0.5}}
\put(8.5,0){\circle*{0.5}}
\put(9.5,0){\circle{0.5}}
\put(10.5,0){\circle{0.5}}
\put(11.5,0){\circle{0.5}}
\put(12.5,0){\circle{0.5}}
\put(13.5,0){\circle{0.5}}
\put(14.5,0){\circle*{0.5}}

\put(15,-0.5){\line(0,1){1}}

\put(15.5,0){\circle*{0.5}}
\put(16.5,0){\circle{0.5}}
\put(17.5,0){\circle{0.5}}
\put(18.5,0){\circle*{0.5}}
\put(19.5,0){\circle*{0.5}}
\put(20.5,0){\circle{0.5}}
\put(21.5,0){\circle*{0.5}}
\put(22.5,0){\circle*{0.5}}
\put(23.5,0){\circle*{0.5}}
\put(24.5,0){\circle*{0.5}}
\put(25.5,0){\circle*{0.5}}

\end{picture} \\
- &  \hspace{-1cm} \setlength{\unitlength}{0.4cm} \begin{picture}(30,0.3)

\put(2.5,0){\ldots}
\put(26,0){\ldots}

\put(4.5,0){\circle{0.5}}
\put(5.5,0){\circle{0.5}}
\put(6.5,0){\circle{0.5}}
\put(7.5,0){\circle{0.5}}
\put(8.5,0){\circle{0.5}}
\put(9.5,0){\circle{0.5}}
\put(10.5,0){\circle*{0.5}}
\put(11.5,0){\circle{0.5}}
\put(12.5,0){\circle*{0.5}}
\put(13.5,0){\circle{0.5}}
\put(14.5,0){\circle{0.5}}

\put(15,-0.5){\line(0,1){1}}

\put(15.5,0){\circle*{0.5}}
\put(16.5,0){\circle{0.5}}
\put(17.5,0){\circle{0.5}}
\put(18.5,0){\circle*{0.5}}
\put(19.5,0){\circle*{0.5}}
\put(20.5,0){\circle{0.5}}
\put(21.5,0){\circle*{0.5}}
\put(22.5,0){\circle*{0.5}}
\put(23.5,0){\circle*{0.5}}
\put(24.5,0){\circle*{0.5}}
\put(25.5,0){\circle*{0.5}}

\end{picture} \\
+ &  \hspace{-1cm} \setlength{\unitlength}{0.4cm} \begin{picture}(30,0.3)

\put(2.5,0){\ldots}
\put(26,0){\ldots}

\put(4.5,0){\circle{0.5}}
\put(5.5,0){\circle{0.5}}
\put(6.5,0){\circle{0.5}}
\put(7.5,0){\circle{0.5}}
\put(8.5,0){\circle{0.5}}
\put(9.5,0){\circle{0.5}}
\put(10.5,0){\circle{0.5}}
\put(11.5,0){\circle*{0.5}}
\put(12.5,0){\circle*{0.5}}
\put(13.5,0){\circle{0.5}}
\put(14.5,0){\circle*{0.5}}

\put(15,-0.5){\line(0,1){1}}

\put(15.5,0){\circle{0.5}}
\put(16.5,0){\circle{0.5}}
\put(17.5,0){\circle{0.5}}
\put(18.5,0){\circle*{0.5}}
\put(19.5,0){\circle*{0.5}}
\put(20.5,0){\circle{0.5}}
\put(21.5,0){\circle*{0.5}}
\put(22.5,0){\circle*{0.5}}
\put(23.5,0){\circle*{0.5}}
\put(24.5,0){\circle*{0.5}}
\put(25.5,0){\circle*{0.5}}

\end{picture} \\
+ &  \hspace{-1cm} \setlength{\unitlength}{0.4cm} \begin{picture}(30,0.3)

\put(2.5,0){\ldots}
\put(26,0){\ldots}

\put(4.5,0){\circle{0.5}}
\put(5.5,0){\circle{0.5}}
\put(6.5,0){\circle{0.5}}
\put(7.5,0){\circle{0.5}}
\put(8.5,0){\circle{0.5}}
\put(9.5,0){\circle{0.5}}
\put(10.5,0){\circle{0.5}}
\put(11.5,0){\circle{0.5}}
\put(12.5,0){\circle*{0.5}}
\put(13.5,0){\circle{0.5}}
\put(14.5,0){\circle*{0.5}}

\put(15,-0.5){\line(0,1){1}}

\put(15.5,0){\circle*{0.5}}
\put(16.5,0){\circle{0.5}}
\put(17.5,0){\circle*{0.5}}
\put(18.5,0){\circle*{0.5}}
\put(19.5,0){\circle*{0.5}}
\put(20.5,0){\circle{0.5}}
\put(21.5,0){\circle{0.5}}
\put(22.5,0){\circle*{0.5}}
\put(23.5,0){\circle*{0.5}}
\put(24.5,0){\circle*{0.5}}
\put(25.5,0){\circle*{0.5}}

\end{picture} \\
- &  \hspace{-1cm} \setlength{\unitlength}{0.4cm} \begin{picture}(30,0.3)

\put(2.5,0){\ldots}
\put(26,0){\ldots}

\put(4.5,0){\circle{0.5}}
\put(5.5,0){\circle{0.5}}
\put(6.5,0){\circle{0.5}}
\put(7.5,0){\circle{0.5}}
\put(8.5,0){\circle{0.5}}
\put(9.5,0){\circle{0.5}}
\put(10.5,0){\circle{0.5}}
\put(11.5,0){\circle{0.5}}
\put(12.5,0){\circle*{0.5}}
\put(13.5,0){\circle{0.5}}
\put(14.5,0){\circle*{0.5}}

\put(15,-0.5){\line(0,1){1}}

\put(15.5,0){\circle*{0.5}}
\put(16.5,0){\circle{0.5}}
\put(17.5,0){\circle{0.5}}
\put(18.5,0){\circle*{0.5}}
\put(19.5,0){\circle*{0.5}}
\put(20.5,0){\circle*{0.5}}
\put(21.5,0){\circle*{0.5}}
\put(22.5,0){\circle*{0.5}}
\put(23.5,0){\circle*{0.5}}
\put(24.5,0){\circle{0.5}}
\put(25.5,0){\circle*{0.5}}
\end{picture}
\end{aligned}
\end{equation} 
\end{Example}

\begin{Comment} For $k >0$, the basis vectors which have non-zero coefficient in $\alpha_k |\lambda,h \rangle$ are exactly those $|\mu,h \rangle$ where $\mu$ is obtained from $\lambda$ by removing a ``ribbon" or ``rim-hook" of length $k$. For $k<0$, they are those $|\mu,h \rangle$ where $\mu$ is obtained from $\lambda$ by adding a ribbon of length $k$. This is part of the reason ribbons appear in the study of Fock space (see e.g. \cite{LLT}).
\end{Comment}

For $k \neq 0$ we have $\psi_m\psi_{m+k}^*=- \psi_{m+k}^* \psi_m$, so $ \pi_\Fock(\alpha_k)  \in \widetilde \Cl$.
Thus Proposition \ref{finding-Bosons} gives an imbedding of $\cW$ into $\widetilde \Cl$. It is natural to ask if one can go the other way and express $\Cl$ in terms of $\cW$. All elements of $\cW$ preserve the subspaces $\Fock^{(h)}$ of $\Fock$, and the generators of $\Cl$ do not preserve these subspaces. So we will need to enlarge $\cW$. It turns out that it suffices to introduce one more operator. 

\begin{Definition}
The shift operator $s: \Fock \rightarrow \Fock$ is defined by $s |\lambda, h \rangle=|\lambda, h+1\rangle$. Note that $s$ corresponds to multiplication by $q$ under the vector space isomorphism $\sigma:\Fock \rightarrow \Bock$. 
\end{Definition}

Introduce the following power series.
In each case all but finitely many terms with negative powers of $z$ act as zero on any fixed $|\lambda, h \rangle,$ so acting on any fixed element of $\Fock$ gives a power series with coefficients in $\Fock$ and only finitely many negative powers of $z$.

\begin{align}
\label{field1} \psi(z) & = \sum_{m \in \bz+\frac12} z^m \pi_F(\psi_m) \\
\psi^*(z)&  = \sum_{m \in \bz+\frac12} z^{-m} \pi_F(\psi_m^*) \\
\Gamma_+(z) & = \exp \sum_{k \in \bz_{>0}} \frac{z^{-k}}{k} \pi_F(\alpha_k) \\
\label{field4} \Gamma_-(z) &= \exp \sum_{k \in \bz_{>0}} \frac{z^k}{k} \pi_F(\alpha_{-k}).
\end{align}

The following is \cite[Theorem 14.10]{Kac}, adjusted slightly to match our conventions. 

\begin{Proposition} \label{finding-Fermions}
As operators on $\Fock$, 
\begin{align*}
 \psi(z) & =  s z^{\ch +\frac12}  \Gamma_-(z) \Gamma_+(z)^{-1},  \\
 \psi^*(z) &=   s^{-1} z^{-\ch+\frac12}  \Gamma_-(z)^{-1} \Gamma_+(z). 
\end{align*}
Here $\ch$ sends a charged partition to its charge, so $z^{\ch}$ acts on $|\lambda,h \rangle$ as multiplication by $z^h$.
\end{Proposition}

Proposition \ref{finding-Fermions}, in principle, expresses the fermionic operators $\psi_m$ and $\psi_m^*$ in terms of the $\alpha_k$. 
Its proof is pretty involved. It proceeds roughly as follows. Kac first shows that the right sides of the equations in Proposition \ref{finding-Bosons} generate an algebra of operators on $\Fock$ which is isomorphic to $\cW$. This then implies the existence of some vector space isomorphism from $\Bock$ to $\Fock$ which satisfies Proposition \ref{finding-Bosons}. One then studies commutation relations between the $\alpha_k$ and the generating functions $\psi(z)$ and $\psi^*(z)$ to prove that this isomorphism also satisfies Proposition \ref{finding-Fermions}. The final step is to show that the isomorphism is given by Proposition \ref{Fock-sym}. So really these notes are basically backwards!

\begin{Example}
Proposition \ref{finding-Fermions} predicts, for instance,
\begin{equation} \label{eq:exBC1}
\begin{aligned}
&[z^{3.5}] s z^{\ch +\frac12}  \Gamma_-(z) \Gamma_+(z)^{-1}  \left(
 \setlength{\unitlength}{0.4cm} \hspace{2.6cm} 
 \begin{picture}(4.5,0.3)

\put(-6,0){$\cdots$}
\put(3,0){$\cdots$}

\put(-4.5,0.3){\circle{0.5}}
\put(-3.5,0.3){\circle{0.5}}
\put(-2.5,0.3){\circle*{0.5}}
\put(-1.5,0.3){\circle{0.5}}
\put(-0.5,0.3){\circle*{0.5}}

\put(0,-0.5){\line(0,1){1.3}}

\put(0.5,0.3){\circle*{0.5}}
\put(1.5,0.3){\circle*{0.5}}
\put(2.5,0.3){\circle*{0.5}}

\end{picture}
\right) =
 \setlength{\unitlength}{0.4cm} \hspace{2.4cm} 
 \begin{picture}(4.5,0.3)

\put(-6,0){$\cdots$}
\put(3,0){$\cdots$}

\put(-4.5,0.3){\circle{0.5}}
\put(-3.5,0.3){\circle*{0.5}}
\put(-2.5,0.3){\circle*{0.5}}
\put(-1.5,0.3){\circle{0.5}}
\put(-0.5,0.3){\circle*{0.5}}

\put(0,-0.5){\line(0,1){1.3}}

\put(0.5,0.3){\circle*{0.5}}
\put(1.5,0.3){\circle*{0.5}}
\put(2.5,0.3){\circle*{0.5}}

\end{picture}.
\end{aligned}
\end{equation}
We will check this by explicitly calculating the left side. 
The only terms in $\Gamma_+(z)^{-1}$ that act without killing this vector are the identity and $-z^{-1} \alpha_1$. So the left side is equal to
$$
\begin{aligned}
&[z^{3.5}]  s z^{\ch +\frac12}  \Gamma_-(z)
(
 \setlength{\unitlength}{0.4cm}
 \hspace{2cm}
 \begin{picture}(4.5,0.3)

\put(-5,0){$\cdots$}
\put(3,0){$\cdots$}

\put(-3.5,0.3){\circle{0.5}}
\put(-2.5,0.3){\circle*{0.5}}
\put(-1.5,0.3){\circle{0.5}}
\put(-0.5,0.3){\circle*{0.5}}

\put(0,-0.5){\line(0,1){1.3}}

\put(0.5,0.3){\circle*{0.5}}
\put(1.5,0.3){\circle*{0.5}}
\put(2.5,0.3){\circle*{0.5}}

\end{picture} 
- z^{-1} 
 \hspace{2cm}
 \begin{picture}(4.5,0.3)

\put(-5,0){$\cdots$}
\put(3,0){$\cdots$}

\put(-3.5,0.3){\circle{0.5}}
\put(-2.5,0.3){\circle{0.5}}
\put(-1.5,0.3){\circle*{0.5}}
\put(-0.5,0.3){\circle*{0.5}}

\put(0,-0.5){\line(0,1){1}}

\put(0.5,0.3){\circle*{0.5}}
\put(1.5,0.3){\circle*{0.5}}
\put(2.5,0.3){\circle*{0.5}}

\end{picture}
).
\end{aligned}
$$
$\Gamma_-(z)$ commutes with $sz^{\text{ch}+\frac12}$, since each $\alpha_k$ does, and $sz^{\text{ch}+\frac12}$ multiplies these elements by $z^{2.5}$ and then shifts all beads one step to the left. This leaves
$$
[z^{3.5}] \left(
 z^{2.5}
 \Gamma_-(z)
(
 \setlength{\unitlength}{0.4cm}
 \hspace{2.4cm}
 \begin{picture}(4.5,0.3)

\put(-6,0){$\cdots$}
\put(3,0){$\cdots$}

\put(-4.5,0.3){\circle{0.5}}
\put(-3.5,0.3){\circle*{0.5}}
\put(-2.5,0.3){\circle{0.5}}
\put(-1.5,0.3){\circle*{0.5}}
\put(-0.5,0.3){\circle*{0.5}}

\put(0,-0.5){\line(0,1){1.3}}

\put(0.5,0.3){\circle*{0.5}}
\put(1.5,0.3){\circle*{0.5}}
\put(2.5,0.3){\circle*{0.5}}

\end{picture} 
)
- z^{1.5}   \Gamma_-(z)(
 \hspace{2cm}
 \begin{picture}(4.5,0.3)

\put(-5,0){$\cdots$}
\put(3,0){$\cdots$}

\put(-3.5,0.3){\circle{0.5}}
\put(-2.5,0.3){\circle*{0.5}}
\put(-1.5,0.3){\circle*{0.5}}
\put(-0.5,0.3){\circle*{0.5}}

\put(0,-0.5){\line(0,1){1.3}}

\put(0.5,0.3){\circle*{0.5}}
\put(1.5,0.3){\circle*{0.5}}
\put(2.5,0.3){\circle*{0.5}}

\end{picture}
)\right).
$$
The degree 1 part of $\Gamma_-(z)$ is $\alpha_{-1}$ and the degree $2$ parts is $\frac{\alpha_{-1}^2}{2}+\frac{\alpha_{-2}}{2}$, so this is 
$$
\alpha_{-1} (
 \setlength{\unitlength}{0.4cm}
 \hspace{2.4cm}
 \begin{picture}(4.5,0.3)

\put(-6,0){$\cdots$}
\put(3,0){$\cdots$}

\put(-4.5,0.3){\circle{0.5}}
\put(-3.5,0.3){\circle*{0.5}}
\put(-2.5,0.3){\circle{0.5}}
\put(-1.5,0.3){\circle*{0.5}}
\put(-0.5,0.3){\circle*{0.5}}

\put(0,-0.5){\line(0,1){1.3}}

\put(0.5,0.3){\circle*{0.5}}
\put(1.5,0.3){\circle*{0.5}}
\put(2.5,0.3){\circle*{0.5}}

\end{picture} 
)
-
(\frac{\alpha_{-1}^2}{2}+\frac{\alpha_{-2}}{2})(
 \hspace{2cm}
 \begin{picture}(4.5,0.3)

\put(-5,0){$\cdots$}
\put(3,0){$\cdots$}

\put(-3.5,0.3){\circle{0.5}}
\put(-2.5,0.3){\circle*{0.5}}
\put(-1.5,0.3){\circle*{0.5}}
\put(-0.5,0.3){\circle*{0.5}}

\put(0,-0.5){\line(0,1){1.3}}

\put(0.5,0.3){\circle*{0.5}}
\put(1.5,0.3){\circle*{0.5}}
\put(2.5,0.3){\circle*{0.5}}

\end{picture}
).
$$
Calculating, 
$$
\begin{aligned}
\alpha_{-1} (
 \setlength{\unitlength}{0.4cm}
 \hspace{2.4cm}
 \begin{picture}(4.5,0.3)

\put(-6,0){$\cdots$}
\put(3,0){$\cdots$}

\put(-4.5,0.3){\circle{0.5}}
\put(-3.5,0.3){\circle*{0.5}}
\put(-2.5,0.3){\circle{0.5}}
\put(-1.5,0.3){\circle*{0.5}}
\put(-0.5,0.3){\circle*{0.5}}

\put(0,-0.5){\line(0,1){1.3}}

\put(0.5,0.3){\circle*{0.5}}
\put(1.5,0.3){\circle*{0.5}}
\put(2.5,0.3){\circle*{0.5}}

\end{picture} 
)= & \hspace{0.47cm}
\setlength{\unitlength}{0.4cm}
 \hspace{2.4cm}
 \begin{picture}(4.5,0.3)

\put(-6,0){$\cdots$}
\put(3,0){$\cdots$}

\put(-4.5,0.3){\circle*{0.5}}
\put(-3.5,0.3){\circle{0.5}}
\put(-2.5,0.3){\circle{0.5}}
\put(-1.5,0.3){\circle*{0.5}}
\put(-0.5,0.3){\circle*{0.5}}

\put(0,-0.5){\line(0,1){1.3}}

\put(0.5,0.3){\circle*{0.5}}
\put(1.5,0.3){\circle*{0.5}}
\put(2.5,0.3){\circle*{0.5}}

\end{picture} \\
&+  \setlength{\unitlength}{0.4cm}
 \hspace{2.4cm}
 \begin{picture}(4.5,0.3)

\put(-6,0){$\cdots$}
\put(3,0){$\cdots$}

\put(-4.5,0.3){\circle{0.5}}
\put(-3.5,0.3){\circle*{0.5}}
\put(-2.5,0.3){\circle*{0.5}}
\put(-1.5,0.3){\circle{0.5}}
\put(-0.5,0.3){\circle*{0.5}}

\put(0,-0.5){\line(0,1){1.3}}

\put(0.5,0.3){\circle*{0.5}}
\put(1.5,0.3){\circle*{0.5}}
\put(2.5,0.3){\circle*{0.5}}

\end{picture} 
\end{aligned}
$$
$$
\begin{aligned}
(\frac{\alpha_{-1}^2}{2}+\frac{\alpha_{-2}}{2})(
 \setlength{\unitlength}{0.4cm}
 \hspace{2.4cm}
 \begin{picture}(4.5,0.3)

\put(-6,0){$\cdots$}
\put(3,0){$\cdots$}

\put(-4.5,0.3){\circle{0.5}}
\put(-3.5,0.3){\circle{0.5}}
\put(-2.5,0.3){\circle*{0.5}}
\put(-1.5,0.3){\circle*{0.5}}
\put(-0.5,0.3){\circle*{0.5}}

\put(0,-0.5){\line(0,1){1.3}}

\put(0.5,0.3){\circle*{0.5}}
\put(1.5,0.3){\circle*{0.5}}
\put(2.5,0.3){\circle*{0.5}}

\end{picture} 
)= & \hspace{0.47cm}
 \setlength{\unitlength}{0.4cm}
 \frac12 \hspace{2.4cm}
 \begin{picture}(4.5,0.3)

\put(-6,0){$\cdots$}
\put(3,0){$\cdots$}

\put(-4.5,0.3){\circle*{0.5}}
\put(-3.5,0.3){\circle{0.5}}
\put(-2.5,0.3){\circle{0.5}}
\put(-1.5,0.3){\circle*{0.5}}
\put(-0.5,0.3){\circle*{0.5}}

\put(0,-0.5){\line(0,1){1.3}}

\put(0.5,0.3){\circle*{0.5}}
\put(1.5,0.3){\circle*{0.5}}
\put(2.5,0.3){\circle*{0.5}}

\end{picture} \\
&+  \setlength{\unitlength}{0.4cm}
 \frac12   \hspace{2.4cm}
\begin{picture}(4.5,0.3)

\put(-6,0){$\cdots$}
\put(3,0){$\cdots$}

\put(-4.5,0.3){\circle{0.5}}
\put(-3.5,0.3){\circle*{0.5}}
\put(-2.5,0.3){\circle*{0.5}}
\put(-1.5,0.3){\circle{0.5}}
\put(-0.5,0.3){\circle*{0.5}}

\put(0,-0.5){\line(0,1){1.3}}

\put(0.5,0.3){\circle*{0.5}}
\put(1.5,0.3){\circle*{0.5}}
\put(2.5,0.3){\circle*{0.5}}

\end{picture} \\
&+
 \setlength{\unitlength}{0.4cm}
 \frac12 \hspace{2.4cm}
 \begin{picture}(4.5,0.3)

\put(-6,0){$\cdots$}
\put(3,0){$\cdots$}

\put(-4.5,0.3){\circle*{0.5}}
\put(-3.5,0.3){\circle{0.5}}
\put(-2.5,0.3){\circle{0.5}}
\put(-1.5,0.3){\circle*{0.5}}
\put(-0.5,0.3){\circle*{0.5}}

\put(0,-0.5){\line(0,1){1.3}}

\put(0.5,0.3){\circle*{0.5}}
\put(1.5,0.3){\circle*{0.5}}
\put(2.5,0.3){\circle*{0.5}}

\end{picture} \\
&-  \setlength{\unitlength}{0.4cm}
 \frac12   \hspace{2.4cm}
\begin{picture}(4.5,0.3)

\put(-6,0){$\cdots$}
\put(3,0){$\cdots$}

\put(-4.5,0.3){\circle{0.5}}
\put(-3.5,0.3){\circle*{0.5}}
\put(-2.5,0.3){\circle*{0.5}}
\put(-1.5,0.3){\circle{0.5}}
\put(-0.5,0.3){\circle*{0.5}}

\put(0,-0.5){\line(0,1){1.3}}

\put(0.5,0.3){\circle*{0.5}}
\put(1.5,0.3){\circle*{0.5}}
\put(2.5,0.3){\circle*{0.5}}

\end{picture} 
\end{aligned}
$$
Subtracting leaves exactly the right side of \eqref{eq:exBC1}. 
\end{Example}

\begin{Comment}
It may seem unfortunate that we need to introduce the power series \eqref{field1}-\eqref{field4}, but these are in fact very useful. Each acts on any $v \in \Fock$ to give a Laurent power series in $z$ with coefficients in $\Fock$, so each of the products below is well-defined as a map from $\Fock$ to Laurent power series in two variables with coefficients in $\Fock$. The following commutation relations follow from results in \cite[Chapter 14]{Kac} (see also \cite[Appendix B2]{OR2}).
\begin{align}
& \Gamma_+(x)\Gamma_-(y) = (1-xy) \Gamma_-(y) \Gamma_+(x) \\
& \Gamma_+(x) \psi(z)= (1-z^{-1}x)^{-1} \psi(z) \Gamma_+(x) \\
& \Gamma_-(x) \psi(z) = (1-xz)^{-1} \psi(z) \Gamma_-(x) \\
& \Gamma_+(x) \psi^*(z)= (1-z^{-1}x) \psi^*(z) \Gamma_+(x) \\
& \Gamma_-(x) \psi^*(z) = (1-xz) \psi^*(z) \Gamma_-(x).
\end{align}
Here $(1-a)^{-1}$ is always expanded as $1+a+a^2+ \cdots$. Equality is interpreted as saying that, once each side is applied to any fixed $v \in \Fock$, all coefficients of the resulting power series agree. These facts have been put to great use. For example, in \cite{OR1} Okounkov and Reshetikhin interpret the $\Gamma_\pm(z)$ as transition functions, and use the above commutation relations to find limit shapes and correlation functions for systems of random plane partitions. 
\end{Comment}

\subsection{$\Fock$ as a representation of $\fgl_{\bz+\frac12}$ and $a_\infty$} \label{F-a}

\begin{Definition}
A Lie-associative map from a Lie algebra $\g$ to an associative algebra $A$ is a map $\sigma$ such that, for all $X, Y \in \g$, $\sigma([X,Y])= \sigma(X) \sigma(Y)-\sigma(Y) \sigma(X)$.
\end{Definition}
An action of a Lie algebra $L$ on a space $V$ is a map $L \rightarrow \End(V)$ which preserves the Lie bracket, which is the same thing as a Lie-Associative map  from $L$ to $\End(V)$. More generally, if an associative algebra $A$ acts on $V$, then any Lie-Associative map $L \rightarrow A$ induces an acton of $L$ on $V$. In the following we use this framework to construct actions of various Lie algebras on $\Fock$ by finding Lie-Associative maps from these algebras to $\Cl$ or the completion $\widetilde \Cl$. 

\begin{Definition}
$M_{\bz+\frac12}$ is the algebra of matrices with rows and columns indexed by $\bz+\frac12$, in which all but finitely many entries are 0. Let $E_{m,n}$ denote the matrix with a single 1 in position $(m,n)$ and zeros everywhere else.
\end{Definition}

\begin{Definition}
$\fgl_{\bz+\frac12}$ is the Lie algebra corresponding to $M_{\bz+\frac12}$. That is, $\fgl_{\bz+\frac12}$ is equal to $M_{\bz+\frac12}$ as a vector space, and the Lie bracket is defined by
\begin{equation*}
[X,Y]= XY-YX.
\end{equation*}
\end{Definition}

\begin{Proposition} \label{prop:glia} (see \cite[\S14.9]{Kac})
There is a Lie-associative embedding of $\fgl_{\bz+\frac12}$ into $\Cl$ given by
$E_{m,n} \rightarrow \psi_m \psi^*_n$, and hence an action of $\fgl_{\bz+\frac12}$  on $\Fock$. 
\end{Proposition}

\begin{Example}
Both $\fgl_{\bz+\frac12}$ and $\Cl$ are associative algebras, so one might hope that the map in Proposition \ref{prop:glia} was a map of associative algebra. However, 
$$E_{12}E_{31} =0 \quad \text{but}  \quad \psi_1\psi_2^* \psi_3\psi_1^*  = \psi_1\psi_3\psi_1^* \psi_2^* \neq 0.$$
Working with these as Lie algebras though, 
$$ E_{12}E_{31} - E_{31}E_{12}  = - E_{32}$$
and
$$ \psi_1\psi_2^* \psi_3\psi_1^* -   \psi_3\psi_1^* \psi_1\psi_2^* = 
 \psi_2^* \psi_3 \psi_1\psi_1^* + \psi_2^*   \psi_3\psi_1^* \psi_1
 = \psi_2^* \psi_3(\psi_1\psi_1^* + \psi_1^* \psi_1)= \psi_2^* \psi_3 =- \psi_3 \psi_2^*,$$
 and the map respects the bracket. One can prove Proposition \ref{prop:glia} by checking relations like this in more generality.
\end{Example}

We would like to extend the action of $\fgl_{\bz+\frac12}$ on $\Fock$ from Proposition \ref{prop:glia} to a larger algebra that contains some matrices with infinitely many non-zero entries. To do this, we will need to first take a central extension. The reason for this extension is explained in Comment \ref{def:cgli} once we've defined a few things.

\begin{Definition} \label{make_fgc} (see \cite[\S7.12]{Kac})
Let $\fgl_{\bz+\frac12}^c$ be the central extension of $\fgl_{\bz+\frac12}$ by a central element $c$ with
Lie bracket defined as follows. We use the notation $\overline X$ to mean the matrix $X$ thought of as an element of $\fgl^c_{\bz+\frac12}$. 
\begin{align*}
[\overline E_{m,n}, \overline E_{p,q}] = 
\begin{cases}
\overline{E_{m,n} E_{p,q}-E_{p,q} E_{m,n}} + \delta_{m,q} \delta_{n,p}  c \quad \text{ if }  m <0 \text{ and } n>0,  \\
\overline{ E_{m,n}  E_{p,q}-E_{p,q}  E_{m,n} }- \delta_{m,q} \delta_{n,p}  c \quad \text{ if }  m > 0 \text{ and } n< 0,  \\
\overline{ E_{m,n}  E_{p,q}- E_{p,q} E_{m,n}} \quad \quad \quad \quad \quad \quad  \text{ if }  m \text{ and } n \text{ have the same sign.} 
\end{cases}
\end{align*}
One can verify by direct calculation that the bracket on $\fgl_{\bz+\frac12}^c$ satisfies the Jacobi identify, so is a Lie algebra, and the quotient by $c=0$ is $\fgl_{\bz+\frac12}$ so it is a central extension.
\end{Definition}

\begin{Comment} \label{trivialize} 
This central extension is trivial since there is an isomorphism of Lie algebras $ \fgl^c_{\bz+\frac12} \rightarrow \fgl_{\bz+\frac12} \oplus \bc c $ given by
\begin{equation*}
\overline E_{m,n} \mapsto
\begin{cases}
 E_{m,n}  \qquad \quad  \text{ if } m \neq n \text{ or } m>0 \\
 E_{m,n}- c \quad  \ \text{ if } m=n \text{ and } m<0.
\end{cases}
\end{equation*}
\end{Comment}

\begin{Proposition} \label{how_ext_acts}  
There is a Lie-associative imbedding  of $\fgl^c_{\bz+\frac12}$ into $\Cl$ defined by
\begin{align}
 \overline E_{m,n} & \mapsto \psi_m \psi^*_n \quad \text{ if } m \neq n \text{ or } m >0, \\ 
 \overline E_{n,n} & \mapsto -\psi^*_n \psi_n \quad \text{ if } n < 0, \\
 c & \rightarrow 1.  
\end{align}
In particular, the action of $\Cl$ on $\Fock$ introduces an action of $\fgl^c_{\bz+\frac12}$ on $\Fock$.
\end{Proposition}

\begin{proof}
Consider the action of  $\fgl_{\bz+\frac12} \oplus \bc c$ on $\Fock$ by allowing $\fgl_{\bz+\frac12}$ to act via the embedding from Proposition \ref{prop:glia}, with $c$ acting as 1. Then use the isomorphism from Comment \ref{trivialize}. For $m \neq n$ or $m >0$, the action of $\overline E_{m,n}$ is immediate. For $n$ negative, $\overline E_{n,n}$ is sent to $E_{n,n}-c$, which acts as  $\psi_n \psi_n^*-1$. By the commutation relation \eqref{Cliff3} this is equal to $-\psi^*_n \psi_n$.
\end{proof}

\begin{Definition} \label{def:cgli} The Lie algebra
$\widetilde \fgl_{\bz+\frac12}$ consists of all infinite matrices $(a_{i,j})_{i,j \in \bz+\frac12}$ such that, for some $k$, $a_{i,j}=0$ whenever $|i-j|>k$. The Lie bracket is the standard bracket for matrices, $[X,Y]=XY-YX$. One can check that matrix multiplication is well-defined on the set of such matrices and hence so is the Lie bracket.
\end{Definition}

\begin{Definition}  \label{def:ca}
$a_\infty$ is the central extension of $\widetilde \fgl_{\bz+\frac12}$ defined as in Definition \ref{make_fgc} by
\begin{align*}
[\overline E_{m,n}, \overline E_{p,q}] = 
\begin{cases}
\overline{E_{m,n} E_{p,q}-E_{p,q} E_{m,n}} + \delta_{m,q} \delta_{n,p}  c \quad \text{ if }  m <0 \text{ and } n>0,  \\
\overline{ E_{m,n}  E_{p,q}-E_{p,q}  E_{m,n} }- \delta_{m,q} \delta_{n,p}  c \quad \text{ if }  m > 0 \text{ and } n< 0,  \\
\overline{ E_{m,n}  E_{p,q}- E_{p,q} E_{m,n}} \quad \quad \quad \quad \quad \quad  \text{ if }  m \text{ and } n \text{ have the same sign.} 
\end{cases}
\end{align*} Any $(a_{mn}) \in a_\infty$ has only finitely many non-zero entries with $m<0$ and $n>0$, so the coefficient of $c$ that appears in any bracket is finite and the central extension is well-defined. 
\end{Definition}

\begin{Proposition} \label{ain}
The imbedding of $\fgl^c_{\bz+\frac12}$ into $\Cl$ extends to an embedding of $a_\infty$ into $\widetilde \Cl$. In particular, $\Fock$ is  a representation of $a_\infty$.
\end{Proposition}

\begin{proof}
This follows from the definition of $\widetilde \Cl$, since the image of any element in $a_\infty$ is of the correct form to lie in the completion $\widetilde \Cl$.
\end{proof}

\begin{Comment}
Essentially, the central extension is needed because, using the action in Proposition \ref{prop:glia}, the action of a sum of the form $\displaystyle \sum_{m <0} x_m E_{m,m}$ on $v \in \Fock$ typically diverges. One fixes this by making a change of basis in $\fgl^c_{\bz+\frac12}$: Instead of working with the elements $E_{m,m}$, work with the elements 
\begin{equation} \label{bE-ing}
\overline E_{m,m}= 
\begin{cases}
E_{m,n}  \quad m \neq n \text{ or } m>0, \\
E_{m,m}-c \quad m = n <0.
\end{cases}
\end{equation}
Then, for any fixed $v \in \Fock$ and $k>0$, there are only finitely many $\overline E_{m,n}$ with $|m-n|<k$ that act on $v$ to give a non-zero term. One can think of $a_\infty$ as infinite linear combinations of these $\overline E_{m,n}$ (along with c) subject to the condition that, for some $k$, the coefficient of $\overline E_{m,n}$ is $0$ whenever $|i-j|>k$. Then the action of $\fgl^c_{\bz+\frac12}$ extends naturally to $a_\infty$. 
The bracket in Definitions \ref{make_fgc} and \ref{def:ca} can be derived by taking the normal bracket of the elements $\overline E_{m,n}$ as in \eqref{bE-ing}. The formulas in Proposition \ref{how_ext_acts} come from the fact that $\psi_m \psi_m^*- 1 = - \psi_m^* \psi_m$.

This discussion also implies that $a_\infty$ is a non-trivial central extension of $\widetilde \fgl_{\bz+\frac12}$, since the action only works with the central extension.
 \end{Comment}

\begin{Comment}  \label{aaa}
The images of the $\alpha_m$ in $\widetilde \Cl$ from Proposition \ref{finding-Bosons} are naturally contained in the image of $a_\infty$; they are matrices with all 1 on one diagonal (other than the main diagonal), and all $0$ everywhere else. One can directly calculate that these matrices have the same commutation relations as the generators of $\cW$, which in fact proves that the map from Proposition \ref{finding-Bosons} gives an action of $\cW$ on $\Fock$. The element $\alpha_0$ mentioned in Comment \ref{com:a0} is the image of $\sum_{m \in \bz+\frac12} \overline E_{m,m}$. 
\end{Comment}

\subsection{$\Fock$ as a representation of $\asl_\ell$ and $\agl_\ell$} \label{F_aff}

It turns out that each $\Fock^{(h)}$ carries an action of $\asl_\ell$. $\Fock^{(h)}$ is not irreducible under this action, but all the irreducible components are very similar; they are isomorphic as representations of the derived algebra $\asl_\ell'$. This is explained by the fact that there is a larger affine algebra $\agl_\ell$ which acts on $\Fock^{(h)}$ and $\Fock^{(h)}$ is irreducible under the action of $\agl_\ell$. 

\begin{Definition}
Let $\fsl_\ell$ be the Lie algebra of $\ell \times \ell$ matrices with trace 0 and $\fgl_\ell$ the Lie algebra of all $\ell \times \ell$ matrices. We use the notation $X_{i,j}$ to denote the matrix with a $1$ in position $(i,j)$ and zeros everywhere else (this is to distinguish $X_{i,j} \in  \fgl_\ell$ from $E_{m,n} \in \fgl_{\bz+\frac12}$). 
\end{Definition}

\begin{Definition} \label{asl-def1}
The affine Lie algebra $\asl_\ell$ is, as a vector space $(\fsl_\ell \otimes_\bc \bc[t,t^{-1}]) \oplus \bc c \oplus \bc d$. The Lie bracket is defined by
\begin{align}
[X \otimes t^n, Y \otimes t^m] & = (XY-YX) \otimes t^{m+n} + n \delta_{n,-m} \mathrm{tr}(XY) c,  
 \\
[d, X \otimes t^n] & = n X \otimes t^n, \\
c & \text{ is central.} 
\end{align}
Here $\mathrm{tr}(XY)$ is the trace of $XY$ (which is also a renormalization of the Killing form).\end{Definition}

\begin{Definition}
$\asl_\ell'$ is the derived algebra of $\asl_\ell$. As a vector space, $\asl'_\ell=(\fsl_\ell \otimes \bc[t,t^{-1}] )\oplus \bc c$.
\end{Definition}

\begin{Proposition} \label{asl-def2} (see \cite[Chapter 7.4 and Exercise 7.11]{Kac}) 
$\asl_\ell$ is isomorphic to the affine Kac-Moody algebra with the $\ell$-node Dynkin diagram 
\vspace{0.2cm}
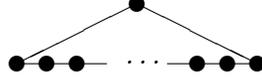
\begin{center}
 \setlength{\unitlength}{0.4cm}

\begin{picture}(6,2)

\put(-1,0){\circle*{0.5}}
\put(0,0){\circle*{0.5}}
\put(1,0){\circle*{0.5}}
\put(5,0){\circle*{0.5}}
\put(6,0){\circle*{0.5}}
\put(7,0){\circle*{0.5}}

\put(3,2){\circle*{0.5}}

\put(3,2){\line(2,-1){4}}
\put(3,2){\line(-2,-1){4}}

\put(-1,0){\line(1,0){1}}
\put(0,0){\line(1,0){1}}
\put(1,0){\line(1,0){1}}
\put(4,0){\line(1,0){1}}
\put(5,0){\line(1,0){1}}
\put(6,0){\line(1,0){1}}

\put(2.6,0){\ldots}

\end{picture}
\end{center}
\vspace{0.1cm}
In particular, 
$\asl_\ell'$ is generated by Chevalley generators $\tilde E_{\bar i}, \tilde F_{\bar i}$ for $i \in \bz/\ell \bz$. Furthermore, the isomorphism can be chosen such that the following hold:
\begin{align}
\label{eq:e1} &\tilde E_{\bar i} \rightarrow X_{i,i+1} \otimes t^0 \quad i=1, \ldots \ell-1 \\
 &\tilde F_{\bar i} \rightarrow X_{i+1, i} \otimes t^0 \quad i=1, \ldots \ell-1 \\
\label{eq:e2} &\tilde E_{\bar 0} \rightarrow X_{\ell,1} \otimes t^1  \\
&\tilde F_{\bar 0} \rightarrow X_{1,\ell} \otimes t^{-1}. 
\end{align}
\end{Proposition}

\begin{Proposition} \label{F_asll} (see \cite[Chapter 14]{Kac}) 
\begin{enumerate}
\item \label{slt1} There is a Lie-associative embedding of $\asl_\ell'/(c-1)$ into $a_\infty$  given by, for any $X_{i,j} \in \text{sl}_\ell
$ and $m \in {\Bbb Z}$, 
 \begin{equation*}
 X_{i,j} \otimes t^m \mapsto \sum_{k \in \bz} \overline E_{i-\frac{1}{2}+k\ell, j-\frac{1}{2}+k\ell+m\ell}.
 \end{equation*}
 Thus, using the embedding $a_\infty \subset \widetilde{Cl}$, $\Fock$ carries an action of $\asl_\ell'$. 

\item \label{slt2} Each $\Fock^{(h)}$ is preserved by the action of $\asl_\ell'$.

\item \label{slt3} Every irreducible component of $\Fock^{(h)}$ under this action is isomorphic to the highest weight representation $V_{\Lambda_{\bar h }}$ as a representation of $\asl_\ell'$.
\end{enumerate}
\end{Proposition}

\begin{proof}
Part \eqref{slt1} can be proven by directly checking that the relations in Definition \ref{asl-def1} hold. Part \eqref{slt2} follows from the fact that the image of any $E_{ij}$ in the map from $a_\infty$ to $\widetilde{Cl}$ preserves each $F^{(h)}$. We delay the proof of part \eqref{slt3} since it follows from the discussion of the $\agl_\ell$ case below. \end{proof}

The action of the Chevalley generators of $\asl_\ell'$ on $\Fock$ can be described combinatorially:

\begin{Proposition} 
\label{comb_sll}
Let $(\lambda, m)$ be a charged partition and fix $\bar i \in \bz/\ell \bz$. Then
\begin{align} \label{eq:Fa}
\tilde E_{\bar i} |\lambda,h \rangle&= \sum_{\small \begin{array}{c} \lambda \backslash \mu \text{ is an} \\ \bar i \text{ colored box } \end{array}} |\mu,h \rangle  & \text{and} \hspace{0.7cm} & \tilde F_{\bar i} |\lambda,h \rangle  &= \sum_{\small \begin{array}{c} \mu \backslash \lambda \text{ is an} \\ \bar i \text{ colored box } \end{array}} |\mu,h \rangle.
\end{align}
Here the boxes are colored as in Figure \ref{partition_bij}.
\end{Proposition}

\begin{proof}
We will explain the case of $\tilde E_{\bar i}$, the case of $\tilde F_{\bar i}$ being similar. Combining formulas \eqref{eq:e1} and \eqref{eq:e2} with Proposition \ref{F_asll},
\begin{equation}
\tilde E_{\bar i} |\lambda,h \rangle = \sum_{k \in {\Bbb Z}} \overline E_{i-\frac{1}{2}+k\ell, i+\frac{1}{2}+k\ell}  |\lambda,h \rangle .
\end{equation}
Using Proposition \ref{how_ext_acts}, this is
\begin{equation} \label{eq:hea}
 \sum_{k \in {\Bbb Z}} \psi_{i-\frac{1}{2}+k\ell}\psi^*_{i+\frac{1}{2}+k\ell}  |\lambda,h \rangle 
\end{equation}
Using the Maya diagram model of Fock space and the bijection in Figure \ref{partition_bij}, 
$ \psi_{i-\frac{1}{2}+k\ell}\psi^*_{i+\frac{1}{2}+k\ell}$ acts an $|\lambda, h \rangle$ by removing a box in position $i+k \ell$, if possible, and it sends $|\lambda, h \rangle$ to zero otherwise. So, using \eqref{eq:hea}, $\tilde E_{\bar i}$ is the sum of all ways to remove $\bar i$ colored boxes, as stated. 
\end{proof}

\begin{Proposition} \label{asll_with_d}
The action of $\asl_\ell'$ on $\Fock$ can be extended to an action of $\asl_\ell$ on $\Fock$ by letting $d$ act on $|\lambda,h \rangle$ as multiplication by 
$$\mathrm{Num}_0 (|\lambda,h \rangle):= \# \text{ squares of } \lambda \text{ colored } c_0,$$
where the coloring is as in Figure \ref{partition_bij}. 
\end{Proposition}

\begin{proof}
Using the expression for the Chevalley generators from Proposition \ref{asl-def2},
it suffices to prove that, as operators on $\Fock$, $d$ commutes with $E_{\bar i}$, $F_{\bar i}$ for $\bar i \neq \bar 0$, that $E_{\bar 0} d = (d+1) E_{\bar 0}$, and that $F_{\bar 0} d  = (d-1) F_{\bar 0}$. The relations involving $E_{\bar i}, F_{\bar i}$ for $\bar i \neq \bar 0$ hold because these operators do not affect the number of $c_{\bar 0}$ colored boxed. The last relations can be seen from the fact that $E_{\bar 0}$ removes a $\bar 0$ colored box from $\lambda$ and $F_{\bar 0}$ adds such a box. 
\end{proof}

The action of $\asl_\ell$ on $\Fock^{(h)}$ can in fact be extended to a larger affine algebra $\agl_\ell$.

\begin{Definition} \label{agl-def1}
The affine Lie algebra $\agl_\ell$ is, as a vector space $(\fgl_\ell \otimes_\bc \bc[t,t^{-1}]) \oplus \bc c \oplus \bc d$. The Lie bracket is defined by
\begin{align}
[X \otimes t^n, Y \otimes t^m] & = (XY-YX) \otimes t^{m+n} + n \delta_{n,-m} \mathrm{tr}(XY) c \\
[d, X \otimes t^n] & = n X \otimes t^n \\
c & \text{ is central.} 
\end{align}
Let
$\agl_\ell'$ denote the Lie-subalgebra $(\fgl_\ell \otimes_\bc \bc[t,t^{-1}]) \oplus \bc c$.
\end{Definition}

\begin{Proposition} \label{goes_to_gl}\hspace{-0.6cm} .
\begin{enumerate}
\item \label{glt1} There is an embedding of $\agl_\ell'/(c-1)$ into $a_\infty$ (and from there into $\widetilde{Cl}$) given by, for $X_{i,j} \in \agl_\ell$, 
 \begin{equation*}
 X_{i,j} \otimes t^m \mapsto \sum_{k \in \bz} \overline E_{i-\frac{1}{2}+k\ell, j-\frac{1}{2}+k\ell+m\ell}. 
 \end{equation*}
 Thus $\Fock$ carries an action of $\agl_\ell'$ which agrees with the action from Proposition \ref{F_asll} on the subalgebra $\asl_\ell'$.
 
 \item \label{glt2} The action can be extended to all of $\agl_\ell$
 by letting $d$ act on $|\lambda, h \rangle$ as multiplication by the number of $c_{\bar 0}$ colored boxes in $\lambda$. 
 
 \item  \label{glt3} Each $\Fock^{(h)}$ is preserved by this action, and is irreducible as a representation of $\agl_\ell$.  

 \end{enumerate}
 
\end{Proposition}

\begin{proof}
Parts \eqref{glt1} and \eqref{glt2} follow exactly as in the $\asl_\ell$ case. Each $\Fock^{(h)}$ is preserved under the action of $\agl_\ell$, so it remains to show that each $\Fock^{(h)}$ is irreducible. For this, we use the fact that $\Fock^{(h)}$ is irreducible as a representation of $\cW$. Thus it suffices to show that for all $k$, $\pi_\Fock \alpha_k$  is in the algebra of operators generated by $\agl_\ell$. For each residue $\bar j$ mod $\ell$, let 
\begin{equation}
C_{\bar j} = \sum_{a=1}^\ell 
\begin{cases}X_{a,a+j} \text{ if } a+j \leq \ell \\
t X_{a,a+j-\ell} \text{ if } a+j > \ell.  
\end{cases}
\end{equation}
Here $j$ is the representative of $\bar j$ in $\{ 0,1, \ldots \ell-1\}$. 
One then directly checks that, for each $k \neq 0$, the actions of $C_{\bar k} \otimes t^{\lfloor k/\ell \rfloor}$ and $\alpha_k$ on $\Fock$ agree.
\end{proof}

\begin{Example}
If $\ell=3$, the elements $C_{\bar j}$ as in the proof of Proposition \ref{goes_to_gl} are  
\begin{equation}
C_{\bar 0}= 
\left(
\begin{array}{ccc}
1&0&0 \\
0&1&0 \\
0&0&1 \\
\end{array}
\right), \quad
C_{\bar 1}= 
\left(
\begin{array}{ccc}
0&1&0 \\
0&0&1 \\
t&0&0 \\
\end{array}
\right), \quad
C_{\bar 2}= 
\left(
\begin{array}{ccc}
0&0&1 \\
t&0&0 \\
0&t&0 \\
\end{array}
\right).
\end{equation}
As in the proof of Proposition \ref{goes_to_gl}, e.g. $\alpha_7$ should correspond to $C_{\bar 1} \otimes t^2$. By definition, 
\begin{equation} 
C_{\bar 1}  \otimes t^2 =
\left(
\begin{array}{ccc}
0&t^2&0 \\
0&0&t^2 \\
t^3&0&0 \\
\end{array}
\right)
= X_{31}t^3+ X_{12}t^2+X_{21} t^2.
\end{equation}
Under the map defined in Proposition \ref{goes_to_gl}, this is sent to 
$$
\sum_{k \in \bz} \overline  E_{3-\frac{1}{2}+3k, 1-\frac{1}{2}+3k+9}
+
\sum_{k \in \bz} \overline  E_{1-\frac{1}{2}+3k, 2-\frac{1}{2}+3k+6}
+
\sum_{k \in \bz} \overline  E_{2-\frac{1}{2}+3k, 3-\frac{1}{2}+3k+6}.
$$
Combining terms and reparametrizing the sums, this is
$$
\sum_{k' \in \bz} \overline  E_{-\frac{1}{2}+k', -\frac{1}{2}+k'+7},
$$
which by Proposition \ref{how_ext_acts} acts on $\Fock$ as
$$
\sum_{k' \in \bz} \psi_{-\frac{1}{2}+k'} \psi^*_{ -\frac{1}{2}+k'+7}.
$$
By Proposition \ref{finding-Bosons} this is exactly the action of $\alpha_7$. 
\end{Example}

One can directly check that the elements $I \otimes t^k \in \agl_\ell'$ commute with all the generators of $\asl_\ell'$. Furthermore, these elements for $k \neq 0$ generate a copy of $\Hi$, where the central element $c_\Hi$ in $\Hi$ is identified with $ \ell c$. This is because, by direct calculation, 
\begin{equation}
[I \otimes t^k, I \otimes t^{-k}]= k \ell  c.
\end{equation}
Together with $I$ itself, these generate all of $\agl_\ell'$, so 
$\agl_\ell'=\asl_\ell' \oplus \Hi \oplus \bc {\bf I}/( \ell c_{\asl_\ell}=c_\Hi)$.

To complete the proof of Proposition \ref{F_asll}\eqref{slt3} use the fact that, by Proposition \ref{goes_to_gl}, each $\Fock^{(h)}$ is an irreducible representation of $\agl_\ell'$, so is a tensor product $V \otimes W $ of an irreducible representation $V$ of $\asl_\ell'$ and an irreducible representation $W$ of $\Hi$.

\subsection{The $q$-deformed Fock space $\Fock_q$ as a representation of $U_q(\asl_\ell)$} \label{MM_section}

We now describe the Misra-Miwa Fock space for $U_q(\asl_\ell)$. This is a representation of $U_q(\asl_\ell)$ originally developed by Misra and Miwa \cite{MM} using work of Hayashi \cite{Hay} (see also \cite[Chapter 10]{Ari}). It can be thought of as a $q$-deformation of the action of $\asl_\ell$ on $\Fock$ described in \S\ref{F_aff}. See e.g. \cite{CP} for the definition of $U_q(\asl_\ell)$. Note that we use $\tilde E_{\bar i}$ and $\tilde F_{\bar i}$ to denote the Chevalley generators, which, in \cite{CP}, are denoted by $X_{i}^+$ and $X_i^-$. 

\begin{Comment}
Caution: $q$ here is unrelated to the $q$ in the definition of bosonic Fock space in \S\ref{ss:BF}. It is however the normal notation so we will accept this conflict. 
\end{Comment}

\begin{Definition}
Let $\Fock_q := \Fock \otimes_\bc \bc(q)$.
\end{Definition}

\begin{Definition} \label{MM_number}
Let $(\lambda,m)$ and $(\mu,m)$ be charged partitions such that $\lambda$ is contained in $\mu$, and $\mu \backslash \lambda$ is a single box. Set
\begin{enumerate}
\item $A_i |\lambda,h \rangle:= \{ c_i \text{ colored boxes } n |\; \lambda \cup n \text{ is a partition} \}.$

\item $R_i|\lambda,h \rangle:= \{c_i \text{ colored boxes } n | \; \lambda \backslash n \text{ is a partition} \}.$

\item $N_i^a(\mu \backslash \lambda):= | \{ n \in A_i(\lambda) |  $n$ \text{ is to the left of } \mu \backslash \lambda \}|
-
 | \{ n \in R_i(\lambda) | \; $n$ \text{ is to the left of } \mu \backslash \lambda \}|
.$

\item $N_i^r(\mu \backslash \lambda):= | \{ n \in A_i(\lambda) |  $n$ \text{ is to the right of  } \mu \backslash \lambda \}|
-
 | \{ n \in R_i(\lambda) | \; $n$ \text{ is to the right of } \mu \backslash \lambda \}|
.$

\end{enumerate}
\end{Definition}

\begin{Theorem} (See \cite[Theorem 10.6]{Ari}) \label{MM_Fock_th}
There is an action of $U_q(\asl_\ell)$ on $\Fock_q$ defined by
\begin{align} \label{eq:Fqa}
\tilde E_{\bar i}|\lambda,h \rangle &:= \sum_{\small \begin{array}{c} \lambda \backslash \mu \text{ is an} \\ i \text{ colored box } \end{array}} q^{-N_i^r(\lambda \backslash \mu )} |\mu,h \rangle & \tilde F_{\bar i}|\lambda,h \rangle &:= \sum_{\small \begin{array}{c} \mu \backslash \lambda \text{ is an} \\ i \text{ colored box } \end{array}} q^{N_i^a(\mu \backslash \lambda )} |\mu,h \rangle.
\end{align}
\end{Theorem}

By comparing \eqref{eq:Fqa} with \eqref{eq:Fa}, when $q$ is set to 1 the $U_q(\asl_\ell)$ action on $\Fock_q$ becomes the $\asl_\ell$ action on $\Fock$ from Proposition \ref{comb_sll}. 
As in the undeformed case, each $F_q^{(h)}$ decomposes as a sum of infinitely many copies of the irreducible representation of highest weight $\Lambda_{\bar h}$.

\end{document}